\documentclass[reqno,11pt]{amsart}
\usepackage[foot]{amsaddr}
\usepackage{bbold}
\usepackage{charter}
\usepackage{enumitem}
\usepackage[margin=1in]{geometry}
\usepackage[colorlinks=true,linktoc=all,linkcolor=blue,citecolor=blue]{hyperref}

\theoremstyle{plain}
\newtheorem{theorem}{Theorem}[section]
\newtheorem*{theorem*}{Theorem}
\newtheorem{proposition}[theorem]{Proposition}
\newtheorem{lemma}[theorem]{Lemma}
\newtheorem{corollary}[theorem]{Corollary}
\theoremstyle{definition}

\newtheorem*{definition*}{Definition}

\numberwithin{equation}{section}
\everymath{\displaystyle}
\allowdisplaybreaks

\newcommand{\N}{\mathbb{N}}
\newcommand{\C}{\mathbb{C}}
\newcommand{\D}{\mathbb{D}}
\newcommand{\modu}[1]{\left|#1\right|}
\newcommand{\set}[2]{\left\{#1 \ \left\vert \ #2 \right. \right\}}
\newcommand{\ls}{\mathcal{L}}
\newcommand{\lsw}{\ls_{\textbf{w}}}
\newcommand{\linf}{L^\infty}
\newcommand{\lunf}{L_\mu^\infty}
\newcommand{\Bloch}{\mathcal{B}}
\newcommand{\norm}[1]{\left\lVert#1\right\rVert}
\newcommand{\lorm}[1]{\left\lVert#1\right\rVert_\ls}
\newcommand{\inorm}[1]{\left\lVert#1\right\rVert_\infty}
\newcommand{\morm}[1]{\left\lVert#1\right\rVert_\mu}

\title[Multiplication Operators]{Multiplication Operators on Weighted Banach Spaces of a Tree}

\author{Robert F.~Allen and Isaac M.~Craig}
\address{Department of Mathematics and Statistics, University of Wisconsin-La Crosse}
\email{rallen@uwlax.edu, craig.isaa@uwlax.edu}

\date{}

\dedicatory{In honor of Flavia Colonna on the occasion of her birthday.}


\begin{document}

\begin{abstract}
We study multiplication operators on the weighted Banach spaces of an infinite tree.  We characterize the bounded and the compact operators, as well as determine the operator norm.  In addition, we determine the spectrum of the bounded multiplication operators and characterize the isometries.  Finally, we study the multiplication operators between the weighted Banach spaces and the Lipschitz space by characterizing the bounded and the compact operators, determine estimates on the operator norm, and show there are no isometries.
\end{abstract}

\maketitle

\section{Introduction}\label{Section:Introduction}
Let $X$ and $Y$ be Banach spaces of functions defined on a set $\Omega$.  For a complex-valued function $\psi$ defined on $\Omega$ the linear operator $M_\psi: X \to Y$ defined by $$M_\psi f = \psi f$$ for all $f \in X$ is called the multiplication operator from $X$ to $Y$ with symbol $\psi$.  To connect the properties of the symbol $\psi$ to the properties of $M_\psi$ is the goal of the study of such operators.

The study of operators with symbol defined on spaces of analytic functions on the open unit disk $\D$ of $\C$ have been studied for many years, and the literature is extensive. There are many such spaces that have been termed classical.  We denote the space of analytic functions from $\D$ to $\C$ by $H(\D)$.  The algebra of bounded analytic functions from $\D$ to $\C$, denoted by $H^\infty(\D)$, is defined to be the set of functions $f \in H(\D)$ for which $$\|f\|_\infty = \sup_{z \in \D}\;|f(z)| < \infty.$$  In addition to $H^\infty$, the Bloch space, denoted $\Bloch(\D)$, is the set of functions $f \in H(\D)$ such that $$\|f\|_\Bloch = |f(0)| + \sup_{z \in \D}\;(1-|z|^2)|f'(z)| < \infty.$$

For the past several years a shift has been made to studying multiplication operators on spaces of functions defined on discrete structures, specifically a tree.  By a \textit{tree\/} $T$ we mean a connected, simply-connected, and locally finite graph.  As a set, we identify the tree with the collection of its vertices. Two vertices
$u$ and $v$ are called \textit{neighbors} if there is an edge $[u,v]$
connecting them, and we use the notation $u\sim v$. A \textit{path} is a finite or infinite sequence of vertices $[v_0,v_1,\dots]$ such that $v_k\sim v_{k+1}$ and  $v_{k-1}\ne v_{k+1}$, for all $k$. A vertex is called \textit{terminal} if it has a unique neighbor. 

Given a tree $T$ rooted at $o$ and a vertex $u\in T$, a vertex $v$ is called a \textit{descendant} of $u$ if the vertex $u$ lies in the unique path from $o$ to $v$. The vertex $u$ is then called an \textit{ancestor} of $v$.  The set of vertices in $T$ other than the root is denoted by $T^*$. For $v \in T^*$, we denote by $v^-$ the unique neighbor which is an ancestor of $v$. For $v\in T$, the set $S_v$ consisting of $v$ and all its descendants is called the \textit{sector} determined by $v$.

The \textit{length} of a finite path $[u=u_0,u_1,\dots,v=u_n]$ (with $u_k\sim u_{k+1}$ for $k=0,\dots, n$) is defined to be the number $n$ of edges connecting $u$ to $v$. The \textit{distance}, $d(u,v)$, between vertices $u$ and $v$ is the length of the unique path connecting $u$ to $v$.  The tree $T$ is a metric space under the distance $d$. Fixing $o$ as the root of the tree, we define the \textit{length} of a vertex $v$, by $|v|=d(o,v)$.  By a \textit{function on a tree} we mean a complex-valued function on the set of its vertices. In this paper, the tree will be assumed to be rooted at a vertex $o$ and infinite.  

The first spaces considered for these types of multiplication operators were discrete analogs of $H^\infty(\D)$ and $\Bloch(\D)$.  The space of bounded functions on a tree $T$ is denoted by $L^\infty(T)$ and defined to be the set of functions $f:T \to\C$ such that $$\|f\|_\infty = \sup_{v \in T}\;|f(v)| < \infty.$$  The Lipschitz space, denoted $\ls(T)$, is the analog of the Bloch space and defined to be  the set of all functions $f:T \to \C$ such that $$\|f\|_\ls = |f(o)| + \sup_{v \in T^*}\;Df(v) < \infty,$$ where $Df(v) = |f(v)-f(v^-)|$ for $v \in T^*$.  The Lipschitz space can be considered the discretization of the Bloch space.  Multiplication operators on $\ls$ were studied in \cite{ColonnaEasley:10} by Colonna and Easley.  

The characterization of the bounded multiplication operators on $\ls$ was used to construct a new space, called the weighted Lipschitz space $\lsw$.  The first author, Colonna, and Easley studied the multiplication operators on $\lsw$ \cite{AllenColonnaEasley:13} and between $\ls$ and $\lsw$ \cite{AllenColonnaEasley:10}.  The  characterization of the bounded operators on $\lsw$ was used to construct yet another new space.  This iterated process constructed a family of spaces called the iterated logarithmic Lipschitz spaces $\ls^{(k)}$ for $k \in \N$, define as the space of functions $f:T \to \C$ such that
$$\|f\|_k = |f(o)| + \sup_{v \in T^*}\;|v|Df(v)\prod_{j=0}^{k-1}\ell_j(|v|) < \infty$$ with, for $x \geq 1$, 
$$\ell_j(x) = \begin{cases}
1 & \text{if } \ j = 0,\\ 1+\log x & \text{if } \ j=1, \\ 1+\log\ell_{j-1}(x) & \text{if } \ j \geq 2.
\end{cases}$$
These spaces are the discrete analogs of the logarithmic Bloch spaces of $\D$, and the multiplication operators between the spaces were studied in \cite{AllenColonnaEasley:12}.  In addition, Colonna and Easley studied multiplication operators between $\ls$ and $L^\infty$ in \cite{ColonnaEasley:12}.  

The purpose of this paper is to introduce a new space of functions on $T$ and to extend the results of \cite{ColonnaEasley:12}.  In addition, this paper will tie together all of the spaces under current study.  Having a complete picture of the multipliers of these spaces will allow researchers to move on to studying the weighted composition operators on these spaces.  A step in this direction was taken by the first author, Colonna, and Easley in \cite{AllenColonnaEasley:14}, where they studied the composition operators on $\ls$.  The results of this paper can be considered discrete versions of results found in \cite{Allen:14}.

\section{Preliminary Results}\label{Section:Spaces}

A weight is a positive real-valued function on a tree $T$.  We define the \textit{weighted Banach space} on $T$ of weight $\mu$, denoted $\lunf(T)$, to be the set of functions $f$ on $T$ such that
\[	 \morm{f} = \sup_{v \in T}\;  \mu(v)|f(v)| < \infty. 		\]
Clearly for $\mu \equiv 1$, we have $\lunf = \linf$.

\begin{theorem}
	The weighted Banach space $\lunf$ is a complex Banach space under $\morm{\cdot}$.
\end{theorem}

\begin{proof}
	It is straightforward to show that $\morm{\cdot}$ endows $\lunf$ with a complex normed linear space structure.  Therefore, it suffices to show that $\lunf$ is complete with respect to $\morm{\cdot}$. Let $\varepsilon > 0$ and suppose $(f_n)$ is a Cauchy sequence in $\lunf$, and $w \in T$ fixed. Since $\mu(w) > 0$, there exists an index $N \in \N$ such that for all $n, m \geq N$, with $m > n$, we have $\morm{f_n - f_m} < \varepsilon\mu(w)$. In fact, 
	\[	 \modu{f_n(w) - f_m(w)} = \dfrac{\modu{\mu(w)f_n(w) - \mu(w)f_m(w)}}{\mu(w)} \leq \dfrac{\morm{f_n - f_m}}{\mu(w)} < \varepsilon. 	\]
	Thus, $(f_n(w))$ is Cauchy in $\C$, and converges to some value $f(w)$. In fact, $(f_n)$ converges pointwise to this function $f$ on $T$. Furthermore, we will show that $f \in \lunf$. Since $(f_n(w))$ converges pointwise to $f(w)$, then for a fixed $w \in T$, there exists an index $N \in \N$ such that for each $n \geq N$, $\modu{f(w) - f_n(w)} < 1/\mu(w)$. By triangle inequality, we have that $\mu(w)\modu{f(w)} < 1 + \mu(w)\modu{f_n(w)}$. Applying the supremum over all $w \in T$, we obtain $\morm{f} < 1 + \morm{f_n}$. Since $(f_n)$ is Cauchy, and thus bounded, we have $\morm{f} < \infty$. Therefore, $f \in \lunf$.
	
	To conclude the proof of the completeness, we need to show that $\morm{f_n - f} \to 0$ as $n \to \infty$.  Arguing by contradiction, assume there exists $\varepsilon > 0$ and a subsequence $(f_{n_j})$ such that $\morm{f_{n_j} - f} \geq \varepsilon$ for all $j \in \N$.  Then for each $j \in \N$, we may choose a vertex $v_{n_j}$ such that $\mu(v_{n_j})\modu{f_{n_j}(v_{n_j}) - f(v_{n_j})} \geq \varepsilon$.  Since $(f_{n_j})$ is Cauchy in $\lunf$, there exists a positive integer $J$ such that for all $j,h \geq J$, and $v \in T$, we have $$\mu(v)\modu{f_{n_j}(v) - f_{n_h}(v)} \leq \morm{f_{n_j}-f_{n_h}} < \frac{\varepsilon}{2}.$$  In particular, for $h \geq J$, we have $$\mu(v_{n_J})\modu{f_{n_J}(v_{n_J}) - f_{n_h}(v_{n_J})} < \frac{\varepsilon}{2}.$$  By the pointwise convergence of $f_{n_h}$ to $f$, for all natural numbers $h$ sufficiently large $$\modu{f_{n_h}(v_{n_J}) - f(v_{n_J})} < \frac{\varepsilon}{2\mu(v_{n_J})},$$ which implies $$\mu(v_{n_J})\modu{f_{n_h}(v_{n_J}) - f(v_{n_J})} < \frac{\varepsilon}{2}.$$  By the triangle inequality, we obtain 
	$$\mu(v_{n_J})\modu{f_{n_J}(v_{n_J}) - f(v_{n_J})} < \varepsilon,$$  contradicting the choice of vertex $v_{n_J}$.  Thus $\morm{f_n - f} \to 0$ as $n \to \infty$, proving the completeness of $\lunf$.
\end{proof}

A Banach space $X$ of complex-valued functions on a set $\Omega$ is said to be a \textit{functional Banach space} if for each $w \in \Omega$, the point evaluation functional is bounded. 

\begin{proposition}
	$\lunf$ is a functional Banach space.
\end{proposition}

\begin{proof}
	Let $f \in \lunf$ with $\morm{f} = 1$, and fix $v \in T$. Observe 
	\[ \modu{f(v)} = \dfrac{\mu(v) \modu{f(v)}}{\mu(v)} \leq \dfrac{\morm{f}}{\mu(v)} = \dfrac{1}{\mu(v)} < \infty. \] 
	Thus, $\lunf$ is a functional Banach space.	
\end{proof}

Colonna and Easley defined the Lipschitz space on a tree in \cite{ColonnaEasley:10}, and proved it to be a functional Banach space, with the point evaluations bounded in the following.

\begin{lemma}\label{Lemma 3.4, ColonnaEasley:10}
	\cite[Lemma 3.4]{ColonnaEasley:10} Let $T$ be a tree and $v \in T$. If $f \in \ls$, then $$\modu{f(v)} \leq \modu{f(o)} + \modu{v} \inorm{Df}.$$ In particular, if $\lorm{f} \leq 1$, then $\modu{f(v)} \leq \modu{v}$ for each $v \in T^*$.
\end{lemma}

A particularly interesting class of functions that live in $\linf$, $\lunf$, and $\ls$ are the characteristic functions.  Let $A \subseteq T$ and define the characteristic function on $A$ to be $$\chi_A(v) = \begin{cases} 1 & \text{if } \ v \in A, \\ 0 & \text{if } \ v \not\in A.\end{cases}$$  In the case that $A$ is a singleton set, $A = \{w\}$, we use the notation $\chi_w$.  In the sections to follow, we will use such characteristic functions and we note their norms here:
$$\begin{aligned}
\inorm{\chi_w} &= \sup_{v \in T}\;|\chi_w(v)| = 1,\\
\morm{\chi_w} &= \sup_{v \in T}\;\mu(v)|\chi_w(v)| = \mu(w),\\
\lorm{\chi_w} &= |\chi_w(o)| + \sup_{v \in T^*} \;D\chi_w(v) = \begin{cases} 2 & \text{if } \ w = o, \\ 1 & \text{if } \ w \neq o.\end{cases}
\end{aligned}$$

For a functional Banach space, the following result is very important in the study of multiplication operators. 

\begin{lemma}\cite[Lemma 11]{DurenRombergShields:69}\label{bounded_mult} Let $X$ be a functional Banach space on the set $\Omega$ and let $\psi$ be a complex-valued function on $\Omega$ such that $M_\psi$ maps $X$ into itself. Then $M_\psi$ is bounded on $X$ and $\modu{\psi(w)} \leq \norm{M_\psi}$ for all $w \in \Omega$. In particular, $\psi$ is bounded.	
\end{lemma}

The following result is inspired by Lemma 2.10 of \cite{Tjani:96}, where the result was proved for Banach spaces of analytic functions on $\D$.  

\begin{lemma}\label{compact_lemma} Let $X,Y$ be two Banach spaces of functions on a tree $T$.  Suppose that
	\begin{enumerate}
		\item[\normalfont{(i)}] the point evaluation functionals of $X$ are bounded,
		\item[\normalfont{(ii)}] the closed unit ball of $X$ is a compact subset of $X$ in the topology of uniform convergence on compact sets,
		\item[\normalfont{(iii)}] $A:X \to Y$ is bounded when $X$ and $Y$ are given the topology of uniform convergence on compact sets.
	\end{enumerate}  Then $A$ is a compact operator if and only if given a bounded sequence $(f_n)$ in $X$ such that $f_n \to 0$ pointwise, then the sequence $(A f_n)$ converges to zero in the norm of $Y$.
\end{lemma}

All spaces mentioned in the previous section satisfy the conditions of the lemma, and we will make extensive use of this result in the study of compact multiplication operators among these spaces.

\section{Multiplication Operators on $\lunf$}\label{Section:WBS}
In this section, we study the multiplication operators acting on the weighted Banach space $\lunf$.  We begin by characterizing the bounded operators and determine the operator norm.  Then we also characterize the compact multiplication operators.  We also determine the spectrum of the bounded operators, and with this we characterize the bounded below multiplication operators.  Finally, we characterize the isometric multiplication operators on $\lunf$.

\subsection{Boundedness and operator norm}
In this section, we characterize the boundedness of the multiplication operators on $\lunf$, as well as determine the operator norm.

\begin{theorem}\label{thm:boundedness1}
	Let $\psi$ be a function on $T$. Then $M_\psi$ is bounded on $\lunf$ if and only if $\psi \in \linf$. Moreover, if $M_\psi$ is bounded on $\lunf$ then $\norm{M_\psi} = \inorm{\psi}$.
\end{theorem}

\begin{proof}
	Suppose $M_\psi : \lunf \to \lunf$ is bounded. Then since $\lunf$ is a functional Banach space on $T$, it follows from Lemma \ref{bounded_mult} that $\psi \in \linf$ and $\inorm{\psi} \leq \norm{M_\psi}$. 
	
	Suppose now that $\psi \in \linf$. For $f \in \lunf$ with $\morm{f}\leq 1$, we observe 	
	$$\morm{M_\psi f} = \sup_{v \in T}\; \mu(v) \modu{\psi(v)f(v)} \leq \inorm{\psi}.$$
	Since $\psi$ is bounded, we have $M_\psi$ is bounded on $\lunf$.  In addition, we obtain $\norm{M_\psi} = \inorm{\psi}$, as desired.
\end{proof}

\subsection{Compactness} 
In this section, we characterize the compact multiplication operators on $\lunf$.

\begin{theorem} Let $M_\psi$ be a bounded multiplication operator on $\lunf$.  Then $M_\psi$ is compact if and only if $\displaystyle\lim_{|v|\to \infty} \psi(v) = 0.$
\end{theorem}

\begin{proof}
	Suppose that $M_\psi$ is compact on $\lunf$. Let $(v_n)$ be a sequence of vertices in $T$ such that $\modu{v_n} \to \infty$, and define for each $n \in \N$ the function $f_n = \dfrac{1}{\mu(v_n)}\chi_{v_n}$. Observe that $\morm{f_n} = 1$ for all $n \in \N$ and $f_n \to 0$ pointwise. From Lemma \ref{compact_lemma}, this implies that $\morm{M_\psi f_n} \to 0$ as $n \to \infty$. Additionally, $$\modu{\psi(v_n)} = \dfrac{\mu(v_n) \modu{\psi(v_n)}}{\mu(v_n)} \leq \sup_{v \in T} \mu(v) \modu{\psi(v)f_n(v) } = \morm{M_\psi f_n} \to 0.$$ Therefore, $\displaystyle \lim_{\modu{v} \to \infty} \psi(v) = 0$.
	
	Next, suppose that $(f_n)$ is a bounded sequence in $\lunf$ converging to 0 pointwise, and $\displaystyle \lim_{\modu{v} \to \infty} \psi(v) = 0$. Let $\varepsilon > 0$ and define $\displaystyle s = \sup_{n \in \N} \;\morm{f_n}$. We may choose $v \in T$ such that $\modu{\psi(v)} < \dfrac{\varepsilon}{s}$.  Observe that $$\mu(v)\modu{\psi(v)}\modu{f_n(v)} \leq \modu{\psi(v)}\morm{f_n} \leq s\modu{\psi(v)}  < \varepsilon.$$ Since $\varepsilon$ is an arbitrary positive number, we have $\mu(v)\modu{\psi(v)f_n(v)} = 0$ for all $v \in T$. Thus, $\morm{M_\psi f_n} \to 0$, and by Lemma \ref{compact_lemma}, we conclude that $M_\psi$ is compact.
\end{proof}

\subsection{Spectrum}
In this section we determine the spectrum and point spectrum of a bounded multiplication operator on $\lunf$. Recall that the \textit{spectrum} of a bounded operator $A$ on a Banach space $X$ is defined as 
\[ 	\sigma(A) = \set{\lambda \in \C}{A - \lambda I \ \text{is not invertible}}, 	\]
where $I$ is the identity operator on $X$. The spectrum of a bounded operator is a nonempty compact subset of $\C$.

The set of eigenvalues $\sigma_p(A)$ of a bounded operator $A$ is called the \textit{point spectrum} of $A$, that is 
\[ 	\sigma_p(A) = \set{\lambda \in \C}{\ker(A - \lambda I) \text{ is not trivial}}.	 \]

\noindent The \textit{approximate point spectrum} of $A$ is defined as the set $\sigma_{ap}(A)$ consisting of all $\lambda \in \C$ corresponding to which for each $n \in \N$ there exists $x_n \in X$ with $\norm{x_n} = 1$ such that $\norm{(A - \lambda I)x_n} \to 0$ as $n \to \infty$. It is clear that 
\begin{equation} \label{eq:1}
	\sigma_p(A) \subseteq \sigma_{ap}(A) \subseteq \sigma(A). 
\end{equation}

\begin{theorem}\label{thm:spectrum}
	Let $M_\psi$ be a bounded multiplication operator on $\lunf$. Then 
	\begin{enumerate} 
		\item[\normalfont{(i)}] $\sigma_p(M_\psi) = \psi(T)$,
		\item[\normalfont{(ii)}] $\sigma(M_\psi) = \sigma_{ap}(M_\psi) = \overline{\psi(T)}$.
	\end{enumerate}
\end{theorem}

\begin{proof}
	Suppose that $\lambda \in \sigma_p(M_\psi)$. Then there exists a nonzero function $f \in \lunf$ with $M_\psi f = \lambda f$. Since $f \not\equiv 0$, there is a vertex $w \in T$ such that $f(w) \neq 0$. Consequently, 
	\[ 	\psi(w)f(w) = (\psi f)(w) = (M_\psi f)(w) = (\lambda f)(w) = \lambda f(w).	 	\] 
	In fact, this implies that $\lambda = \psi(w)$. Therefore, $\lambda \in \psi(T)$, and so $\sigma_p(M_\psi) \subseteq \psi(T)$. 
	
	Suppose now that $\lambda \in \psi(T)$. Then there exists a vertex $w \in T$ such that $\lambda = \psi(w)$. Recall that for the characteristic function $\chi_w$ is in $\lunf$. Observe, 
	\[ 	(M_\psi \chi_w)(w) = \psi(w)\chi_w(w) = \lambda\chi_w(w) = (\lambda\chi_w)(w). 	\] 
	Furthermore $(M_\psi \chi_w - \lambda \chi_w)(w) = 0$, and $(M_\psi \chi_w - \lambda \chi_w)(v) = 0$ for all $v \neq w$. Since $\chi_w \not\equiv 0$ and $(M_\psi \chi_w - \lambda \chi_w) \equiv 0$, it follows that $\lambda \in \sigma_p(M_\psi)$. Therefore, $\psi(T) \subseteq \sigma_p(M_\psi)$, thus concluding $\sigma_p(M_\psi) = \psi(T)$.
	
	We next show that $\sigma(M_\psi) = \overline{\psi(T)}$. Since $\psi(T) = \sigma_p(M_\psi) \subseteq \sigma(M_\psi)$, then passing to the closure we obtain $\overline{\psi(T)} \subseteq \sigma(M_\psi)$. Suppose now that $\lambda \notin \overline{\psi(T)}$. Then there exists a $c > 0$ such that $\modu{\psi(v) - \lambda} \geq c$ for all $v \in T$. Define the function $\varphi_\lambda$ by $\varphi_\lambda(v) = (\psi(v) - \lambda)^{-1}$. Observe that 
	\[ 	\modu{\varphi_\lambda(v)} = \dfrac{1}{\modu{\psi(v) - \lambda}} \leq \dfrac{1}{c}, 	\]
	and so $\varphi_\lambda$ is bounded. Thus, by Theorem \ref{thm:boundedness1}, we have that $M_{\varphi_\lambda}$ is bounded on $\lunf$. Since $M_{\varphi_\lambda} = M_{(\psi - \lambda)^{-1}} = M^{-1}_{\psi - \lambda}$, it follows that $M^{-1}_{\psi - \lambda}$ is bounded on $\lunf$. Observe that 
	\[ 	(M^{-1}_{\psi - \lambda})^{-1}f = \left( \dfrac{1}{\psi - \lambda} \right)^{-1}f = M_{\psi - \lambda}f, 	\] 
	so that we have 
	\[	M_{\psi - \lambda} \left(M_{\psi - \lambda}^{-1} f\right) = (\psi - \lambda)\left(\dfrac{1}{\psi - \lambda} f \right) = f. 	\]
	Consequently, $M_{\psi - \lambda}$ is invertible on $\lunf$, which proves that $\lambda \notin \sigma(M_\psi)$. 
	
	Therefore, it has been shown that $\overline{\psi(T)} = \sigma(M_\psi)$. Furthermore, it is well known (\cite{Conway:90}, Proposition 6.7) that 
	\[ 	\partial \sigma(A) \subseteq \sigma_{ap}(A). 	\] 
	From this and (\ref{eq:1}) we have that $\sigma_{ap}(T) = \overline{\psi(T)}$, as desired.
\end{proof}

Recall that a bounded operator $A$ on a Banach space $X$ is {\it bounded below} if there exists a positive constant $C$ such that $\norm{Ax} \geq C\norm{x}$ for each $x \in X$.  The following result relates the notion of approximate point spectrum of a bounded operator on a Banach space to the notion of an associated operator that is bounded below.

\begin{proposition}\cite[Proposition 6.4]{Conway:90}\label{prop:spectrum}
	For a bounded operator $A$ on a Banach space and for $\lambda \in \C$, the following statements are equivalent: 
	\begin{enumerate}
		\item[\normalfont{(i)}] $\lambda \notin \sigma_{ap}(A)$.
		\item[\normalfont{(ii)}] $A - \lambda I$ is injective and has closed range.
		\item[\normalfont{(iii)}] $A - \lambda I$ is bounded below.	
	\end{enumerate}
\end{proposition}

From Proposition \ref{prop:spectrum} and Theorem \ref{thm:boundedness1}, it follows that if $M_\psi$ is a bounded multiplication operator on $\lunf$, then $M_\psi$ is bounded below if and only if $0 \notin \overline{\psi(T)}$. We deduce the following result.

\begin{corollary}
	The bounded operator $M_\psi$ on $\lunf$ is bounded below if and only if $$\displaystyle\inf_{v \in T}\; |\psi(v)| > 0.$$
\end{corollary}

\subsection{Isometries}
In this section, we characterize the isometric multiplication operators on $\lunf$.

Given Banach spaces $X$ and $Y$, recall that a linear operator $A: X \to Y$ is an {\it isometry} if \[ \norm{Ax} = \norm{x} \hskip10pt \text{for all } x \in X. \] In this section, we show the isometric multiplication operators on $\lunf$ are induced by a particular class of bounded symbol.

\begin{theorem}
	The multiplication operator $M_\psi$ on $\lunf$ is an isometry if and only if $\psi$ is a function of modulus 1. 
\end{theorem}

\begin{proof}
	Suppose that $\psi$ is a function of modulus 1 and let $f \in \lunf$. It follows that \[ \morm{M_\psi f} = \sup_{v \in T}\; \mu(v)\modu{\psi(v)f(v)} = \sup_{v \in T}\; \mu(v)\modu{f(v)} = \morm{f}, \] and thus $M_\psi$ is an isometry. 
	
	Suppose now that $M_\psi$ is an isometry.  For $w \in T$, observe $$\mu(w) = \morm{\chi_w} = \morm{M_\psi\chi_w} = \mu(w)|\psi(w)|.$$  Since $\mu(w) >0$, it must be the case that $|\psi(w)| = 1$.  Thus, $\psi$ is a function of modulus 1.
\end{proof}

\section{Multiplication Operators from $\ls$ to $\lunf$}\label{Section:ls2lunf}
In the next two sections, we study the connections between $\lunf$ and the Lipschitz space $\ls$.  Specifically, in this section we study the multiplication operators from $\ls$ to $\lunf$.  We characterize the bounded operators and determine estimates on their operator norm.  In addition, we characterize the compact operators as well as determine there are no isometries among the multiplication operators.

\subsection{Boundedness and operator norm} In this section, we characterize the bounded multiplication operators from $\ls$ to $\lunf$ as well as determine operator norm estimates.

For a function $\psi$ on $T$, define $$\eta_\psi = \sup_{v \in T}\; \mu(v) \modu{v} \modu{\psi(v)}.$$  Note that $\eta_\psi = 0$ if and only if $\psi$ is identically 0 on $T^*$.

\begin{theorem}\label{Bounded, ls to lunf}
	Let $\psi$ be a function on $T$. Then $M_\psi: \ls \to \lunf$ is bounded if and only if $\eta_\psi$ is finite. Furthermore, for such bounded multiplication operators, $$\max\left\{\dfrac{1}{2}\mu(o)\modu{\psi(o)}, \eta_\psi\right\} \leq \norm{M_\psi} \leq \max\left\{\mu(o)\modu{\psi(o)}, \eta_\psi\right\}.$$
\end{theorem}
\begin{proof}
	Suppose $M_\psi$ is bounded. For the function $f(v) = \modu{v}$, note that $\lorm{f} = 1$ and observe that 
	\begin{equation}\label{thisisequation}
		\sup_{v \in T} \mu(v)\modu{v}\modu{\psi(v)} = \morm{M_\psi f} \leq \norm{M_\psi}.
	\end{equation}  From the boundedness of $M_\psi$, we have that $\eta_\psi$ is finite.  For the function $f = \frac12\chi_o,$ note that $\lorm{f}=1$ and observe
	\begin{equation}\label{eqat}
		\frac12\mu(o)\modu{\psi(o)} = \morm{M_\psi \chi_o} \leq \norm{M_\psi}.
	\end{equation}
	From (\ref{thisisequation}) and (\ref{eqat}), we obtain $\max\left\{\frac12\mu(o)|\psi(o)|,\eta_\psi\right\} \leq \norm{M_\psi}$.
	
	Suppose now that $\eta_\psi$ is finite. Let $g \in \ls$ with $\lorm{g} = 1$. Observe that $\modu{g(o)} \leq 1$, and so \begin{equation}\mu(o)\modu{\psi(o) g(o)} \leq \mu(o)\modu{\psi(o)}.\label{4.1.1}\end{equation} Furthermore, by Lemma \ref{Lemma 3.4, ColonnaEasley:10}, for $v \neq o$, we have $\modu{g(v)} \leq \modu{v}$ so that \begin{equation}\mu(v) \modu{\psi(v) g(v)} \leq \mu(v) \modu{v} \modu{\psi(v)}.\label{4.1.2}\end{equation} Taking the supremum over all $v \in T^*$, we see that $\displaystyle \sup_{v \in T^*} \mu(v) \modu{\psi(v) g(v)} \leq \eta_\psi$. Finally, from (\ref{4.1.1}), (\ref{4.1.2}), and by taking the supremum over all $g \in \ls$ with $\lorm{g} = 1$, the boundedness of $M_\psi$ is established. We conclude that $\norm{M_\psi} \leq \max\left\{\mu(o)|\psi(o)|,\eta_\psi\right\}$, as desired.  
\end{proof}

\subsection{Compactness}
In this section we characterize the compact multiplication operators from $\ls$ to $\lunf$.

We first show a particular class of symbol $\psi$ induces compact multiplication operators from $\ls$ to $\lunf$.

\begin{lemma}\label{compactness_special_case}
	Let $\psi$ be a function on $T$ such that $\eta_\psi = 0$.  Then $M_\psi:\ls \to \lunf$ is compact.
\end{lemma}

\begin{proof} Recall $\eta_\psi = 0$ if and only if $\psi$ identically 0 on $T^*$.  By Theorem \ref{Bounded, ls to lunf} we have $M_\psi:\ls \to \lunf$ is bounded.  Also, if $\psi$ is identically 0 on $T$, then $M_\psi$ is clearly compact.  So suppose $\psi$ is identically 0 on $T^*$ with $\psi(o) \neq 0$ and let $(f_n)$ be a bounded sequence in $\ls$ converging to 0 pointwise on $T$.  By Lemma \ref{compact_lemma}, it suffices to show $\morm{M_\psi f_n} \to 0$ as $n \to \infty$.  Observe 
	$$\morm{M_\psi f_n} = \sup_{v \in T}\;\mu(v)\modu{\psi(v)f_n(v)} = \mu(o)\modu{\psi(o)}\modu{f_n(o)} \to 0$$ as $n \to \infty$ since $f_n \to 0$ pointwise on $T$.  Thus $M_\psi$ is compact.
\end{proof}

\begin{theorem}\label{Compactness, ls to lunf}
	Let $M_\psi: \ls \to \lunf$ be a bounded multiplication operator. Then $M_\psi$ is compact if and only if $\displaystyle \lim_{\modu{v} \to \infty} \mu(v)\modu{v} \modu{\psi(v)} = 0$.
\end{theorem}

\begin{proof}
	Suppose that $M_\psi$ is compact. Let $(v_n)$ be any sequence in $T$ with $\modu{v_n}\geq 1$ and $\modu{v_n}$ increasing without bound. It is sufficient to show that the sequence $\left(\mu(v_n)\modu{v_n}\modu{\psi(v_n)}\right)$ converges to 0 as $n$ tends to infinity. For each $n \in \N$, define
	
	$$f_n(v) = \begin{cases} 0 & \text{if } \ \modu{v} < \dfrac{\modu{v_n}}{2}, \\ 2\modu{v} - \modu{v_n} & \text{if } \ \dfrac{\modu{v_n}}{2} \leq \modu{v} < \modu{v_n}, \\ \modu{v_n} & \text{if } \ \modu{v_n} \leq \modu{v}. \end{cases}$$
	
	\noindent Since $(v_n)$ increases without bound, for each $v \in T$ we can find a sufficiently large $N \in \N$, such that for all $n \geq N$ we have $\modu{v} < \dfrac{\modu{v_n}}{2}$. Thus, $f_n$ converges to 0 pointwise on $T$. Furthermore, since $f_n(o) = 0$, we have $\displaystyle \lorm{f_n} = \sup_{v \in T^*}\; Df_n(v)$. Consider the case where $\dfrac{\modu{v_n}}{2} < \modu{v} \leq \modu{v_n}$, noting that $Df_n = 0$ in all other cases. Observe that for $\modu{v} = \modu{v_n}$ we have
	$$Df_n(v) = \Big\vert \modu{v_n} - (2\modu{v} - 2 - \modu{v_n})\Big\vert = 2,$$
	and when $\dfrac{\modu{v_n}}{2} < \modu{v} < \modu{v_n}$, we have 
	$$Df_n(v) = \Big\vert 2\modu{v} - \modu{v_n} - (2 \modu{v} - 2 - \modu{v_n})\Big\vert = 2.$$ It follows that $\lorm{f_n} = 2$. In fact, from the compactness of $M_\psi$ and Lemma \ref{compact_lemma} $$\mu(v_n)\modu{v_n}\modu{\psi(v_n)} = \mu(v_n)\modu{\psi(v_n) f_n(v_n)} \leq \morm{\psi f_n} \to 0$$ as $n \to \infty$. Consequently, the sequence $\left(\mu(v_n)\modu{v_n}\modu{\psi(v_n)}\right)$ converges to 0 as $n \to \infty$.
	
	Suppose $\displaystyle \lim_{\modu{v} \to \infty} \mu(v)\modu{v}\modu{\psi(v)} = 0$ and $\eta_\psi \neq 0$ ($\psi$ for which $\eta_\psi = 0$ induce compact multiplication operators by Lemma \ref{compactness_special_case}). By Lemma \ref{compact_lemma}, it is sufficient to show if $(f_n)$ is any bounded sequence in $\ls$ converging to 0 pointwise on $T$, then $\morm{\psi f_n} \to 0$ as $n$ tends to $\infty$. Suppose that $(f_n)$ is such a sequence with $s = \displaystyle\sup_{n \in \N}\;\lorm{f_n}$. For convenience assume each $f_n$ is not identically 0, and observe that for all $n \in \N$, the function $g_n = f_n/\lorm{f_n}$ is in $\ls$ since $\lorm{g_n} = 1$. Thus, by Lemma \ref{Lemma 3.4, ColonnaEasley:10} for all $n \in \N$ we have $\modu{g_n(v)} \leq \modu{v}$ for all $v \in T^*$, so we obtain $$\modu{f_n(v)} \leq \modu{v} \lorm{f_n} \leq s \modu{v}.$$ In fact, $\mu(v)\modu{\psi(v) f_n(v)} \leq s \mu(v)\modu{v} \modu{\psi(v)}$ for all $v \in T^*$. Let $\varepsilon > 0$. By our assumption, there exists an $M \in \N$ such that $\mu(v)\modu{v} \modu{\psi(v)} < \dfrac{\varepsilon}{s}$ whenever $\modu{v} > M$. Consequently, for all $v \in T$ with $\modu{v} > M$, we have $$\mu(v)\modu{\psi(v)f_n(v)} < \varepsilon.$$
	
	Taking into consideration the case where $\modu{v} \leq M$, since $(f_n)$ converges to 0 pointwise on $T$, it converges uniformly on the set $\{v \in T : \modu{v} \leq M\}$. That is, with $\eta_\psi \neq 0$ there exists an $N \in \N$ such that $\modu{f_n(v)} < \dfrac{\varepsilon}{\eta_\psi}$ for all $n \geq N$ and all $v \in T$ with $\modu{v} \leq M$. Consequently, $$\mu(v)\modu{\psi(v) f_n(v)} < \dfrac{\varepsilon \mu(v) \modu{\psi(v)}}{\eta_\psi} < \varepsilon$$  for all $n \geq N$ and for all $v \in T$ such that $\modu{v} \leq M$. 
	
	Therefore, for all $\varepsilon > 0$ and all $v \in T$, there exists an $N \in \N$ such that $\mu(v)\modu{\psi(v) f_n(v)} < \varepsilon$ for all $n \geq N$. Taking the supremum over all $v \in T$, it follows that $\morm{\psi f_n} < \varepsilon$ for all $n \geq N$. Letting $\varepsilon \to 0$, we conclude that $\morm{\psi f_n} \to 0$ as $n$ tends to $\infty$.
\end{proof}

\subsection{Isometries}
In this section we determine there are no isometric multiplication operators from $\ls$ to $\lunf$.

\begin{theorem}\label{Isometries, ls to lunf}
	There are no isometric multiplication operators $M_\psi$ from $\ls$ to $\lunf$.
\end{theorem}

\begin{proof}
	Assume that $M_\psi: \ls \to \lunf$ is an isometry. For $v \in T^*$, taking the characteristic function $\chi_v$, it follows that $$ \mu(v)\modu{\psi(v)} = \morm{M_\psi \chi_v} = \lorm{\chi_v} = 1.$$ Thus for all $v \in T^*$, we have 
	\begin{equation}\label{this equation}
		\mu(v)\modu{v} \modu{\psi(v)} = \modu{v}.
	\end{equation}
	Since $M_\psi$ is bounded, Theorem \ref{Bounded, ls to lunf} implies that $\displaystyle \sup_{v \in T} \mu(v)\modu{v} \modu{\psi(v) }$ is finite. However, this contradicts (\ref{this equation}) since by taking $\modu{v} \to \infty$, we would have $ \mu(v)\modu{v} \modu{\psi(v)} \to \infty$. Therefore, $M_\psi$ is not an isometry.
\end{proof}

\section{Multiplication Operators from $\lunf$ to $\ls$}
In this section we study multiplication operators from $\lunf$ to $\ls$.  As in the previous section, we characterize the bounded and compact operators, determine bounds on the operator norm of the bounded operators, and determine there are no isometric multiplication operators.

\subsection{Boundedness and operator norm}
We characterize the bounded multiplication operators and determine the operator norm.

For a function $\psi$ on $T$, define $$\varpi_\psi = \sup_{v \in T} \dfrac{\modu{\psi(v)}}{\mu(v)}.$$

\begin{theorem}\label{Bounded, lunf to ls}
	Let $\psi$ be a function on $T$. Then $M_\psi: \lunf \to \ls$ is bounded if and only if $\varpi_\psi$ is finite. Furthermore, for such a bounded $M_\psi$, $$\varpi_\psi \leq \norm{M_\psi} \leq 3\varpi_\psi.$$
\end{theorem}

\begin{proof}
	Suppose $M_\psi$ is bounded. The function $f_o = \dfrac{1}{\mu(o)}\chi_o$ is in $\lunf$ with $\morm{f_o} = 1$.  Observe
	\begin{equation}\dfrac{\modu{\psi(o)}}{\mu(o)} = \dfrac{1}{2} \lorm{M_\psi f_o} \leq \norm{M_\psi}.\label{5.1.1}\end{equation} Furthermore, for any vertex $v \in T^*$, the function $f_v = \dfrac{1}{\mu(v)}\chi_v$ is also in $\lunf$ with $\morm{f_v} = 1$.  Likewise, \begin{equation}\dfrac{\modu{\psi(v)}}{\mu(v)} = \lorm{M_\psi f_v} \leq \norm{M_\psi}.\label{5.1.2}\end{equation} From (\ref{5.1.1}), (\ref{5.1.2}), and taking the supremum over all $v \in T$, we obtain $\varpi_\psi \leq \norm{M_\psi}$. Thus, the boundedness of $M_\psi$ implies that $\varpi_\psi$ is finite.
	
	Continuing, suppose that $\varpi_\psi$ is finite. Let $f \in \lunf$ with $\morm{f}\leq 1$, and observe that 
	\begin{align*}
		\lorm{M_\psi f} & = \modu{\psi(o)f(o)} + \sup_{v \in T^*}\modu{ \psi(v)f(v) - \psi(v^-)f(v^-)} \\
		& \leq \modu{\psi(o)f(o)} + 2\sup_{v \in T^*}\;\modu{\psi(v)f(v)}\\
		& = \modu{\dfrac{\psi(o)}{\mu(o)} \mu(o)f(o)} + 2\sup_{v \in T^*} \modu{\dfrac{\psi(v)}{\mu(v)}\mu(v)f(v)} \\
		& \leq 3\varpi_\psi.
	\end{align*}
	Therefore, we have that $M_\psi$ is bounded. Furthermore, we conclude $\varpi_\psi \leq \norm{M_\psi} \leq 3\varpi_\psi$, as desired. 
\end{proof}

\subsection{Compactness}
In this section we characterize the compact multiplication operators from $\lunf$ to $\ls$.

\begin{theorem}
	Let $M_\psi: \lunf \to \ls$ be a bounded multiplication operator. Then $M_\psi$ is compact if and only if $\displaystyle \lim_{\modu{v} \to \infty} \dfrac{\modu{\psi(v)}}{\mu(v)} = 0$.
\end{theorem}

\begin{proof}
	Suppose $M_\psi$ is compact, and let $(v_n)$ be a sequence in $T$ with $|v_n| \to \infty$ as $n \to \infty$. For $n \in \N$ define 
	
	$$f_n(v) = \begin{cases}  0 & \text{if } \ \modu{v} < \modu{v_n},\\ 1/\mu(v) & \text{if } \ \modu{v_n} \leq \modu{v}, \end{cases}$$
	and observe that $\morm{f_n} = 1$. Since $|v_n| \to \infty$ as $n \to \infty$, for a fixed $w \in T$ we may choose an index $N \in \N$ sufficiently large such that whenever $n \geq N$ we have $\modu{w} < \modu{v_n}$. Consequently, $f_n$ converges to 0 pointwise on $T$. By Lemma \ref{compact_lemma}, the compactness of $M_\psi$ implies $$\dfrac{\modu{\psi(v_n)}}{\mu(v_n)} = D(\psi f_n)(v_n) \leq \lorm{M_\psi f_n} \to 0$$ as $n \to \infty$. 
	
	Suppose now that $\displaystyle\lim_{|v|\to\infty}\dfrac{\modu{\psi(v)}}{\mu(v)} = 0$ and disregard the case that $\psi$ is identically 0 on $T$, since the associated multiplication operator is compact.  Let $(f_n)$ be a bounded sequence in $\lunf$ converging to 0 pointwise on $T$ and define $s = \displaystyle\sup_{n \in \N}\;\morm{f_n}$.  Given $\varepsilon > 0$, there exists an $N \in \N$ such that $\dfrac{\modu{\psi(v)}}{\mu(v)} < \dfrac{\varepsilon}{2s}$ for all $v \in T$ with $|v| \geq N$.  Thus, for all $v$ with $|v| > N$ and all $n \in \N$, we have
	$$\begin{aligned}
	D(\psi f_n)(v) &= \modu{\psi(v)f_n(v) - \psi(v^-)f_n(v^-)}\\
	&= \modu{\dfrac{\psi(v)}{\mu(v)}\mu(v)f_n(v) - \dfrac{\psi(v^-)}{\mu(v^-)}\mu(v^-)f_n(v^-)}\\
	&\leq \dfrac{\modu{\psi(v)}}{\mu(v)}\mu(v)\modu{f_n(v)} + \dfrac{\modu{\psi(v^-)}}{\mu(v^-)}\mu(v^-)\modu{f_n(v^-)}\\
	&\leq \left(\dfrac{\modu{\psi(v)}}{\mu(v)} + \dfrac{\modu{\psi(v^-)}}{\mu(v^-)}\right)\morm{f_n}\\
	&< \varepsilon. 
	\end{aligned}$$
	
	Now consider the case where $\modu{v} \leq N$.  Since the sequence $(f_n)$ converges to 0 pointwise on $T$, then $(f_n)$ converges uniformly to 0 on the finite set $\{v \in T : \modu{v} \leq N\}$.  Defining $m = \displaystyle\max_{|v| \leq N}\;\mu(v)$, there exists an index $M \in \N$ such that $\modu{f_n(v)} < \dfrac{\varepsilon}{2m\varpi_\psi}$ for all $v \in T^*$ with $\modu{v} \leq N$ and all $n \geq M$.  Consequently, 
	$$\begin{aligned}
	D(\psi f_n)(v) &= \modu{\psi(v)f_n(v) - \psi(v^-)f_n(v^-)} \\
	&\leq \modu{\dfrac{\psi(v)}{\mu(v)}\mu(v)f_n(v)} + \modu{\dfrac{\psi(v^-)}{\mu(v^-)}\mu(v^-)f_n(v^-)} \\
	&\leq m\varpi_\psi \left(\modu{f_n(v)} + \modu{f_n(v^-)}\right) \\
	&< \varepsilon.
	\end{aligned}$$
	
	Thus, $D(\psi f_n)(v) < \varepsilon$ for all vertices $v \neq o$ and all $n \geq M$. Since $(f_n)$ converges to 0 pointwise on $T$, then $\psi f_n(o) \to 0$ as $n \to \infty$. Therefore, we have that $\lorm{\psi f_n} \to 0$ as $n \to \infty$. It is concluded from Lemma \ref{compact_lemma} that $M_\psi$ is compact.
\end{proof}

\subsection{Isometries}
As was the case in Section \ref{Section:ls2lunf}, we determine there are no isometric multiplication operators from $\lunf$ to $\ls$.

\begin{theorem}
	\label{Isometris, lunf to ls}
	There are no isometric multiplication operators on $M_\psi: \lunf \to \ls$.
\end{theorem}

\begin{proof}
	Assume that $M_\psi$ is an isometry from $\lunf$ to $\ls$. Beginning with the characteristic function $\chi_o$, we have 
	\begin{equation}\label{anequation}
		\mu(o) = \morm{\chi_o} = \lorm{M_\psi \chi_o} = 2 \modu{\psi(o)}.
	\end{equation} 
	Moreover, taking the characteristic function $\chi_v$ for $v \in T^*$, we obtain 
	\begin{equation}\label{eq}
		\mu(v) = \morm{\chi_v} = \lorm{M_\psi \chi_v} = \modu{\psi(v)}.
	\end{equation}
	Finally, define $f = 1/\mu$ on $T$. Observe that $f \in \lunf$ with $\morm{f} = 1$. In fact, since $M_\psi$ is assumed to be an isometry, from (\ref{anequation}) and (\ref{eq}) we obtain $$1 = \morm{f} = \lorm{M_\psi f} = \dfrac{\modu{\psi(o)}}{\mu(o)} + \sup_{v \in T^*} \modu{\dfrac{\psi(v)}{\mu(v)}-\dfrac{\psi(v^-)}{\mu(v^-)}} = \dfrac12.$$  Therefore, $M_\psi$ is not an isometry.
\end{proof}

\bibliographystyle{amsplain}
\bibliography{references.bib}
\end{document}